\documentclass[reqno, 11pt]{amsart}%
\usepackage{amsaddr} 
\usepackage{amssymb}
\usepackage[nesting]{hyperref}
\usepackage[pdftex]{graphicx}
\usepackage{listings}
\usepackage{multirow}
\usepackage{placeins}
\usepackage{color}
\usepackage{subfigure}
\usepackage{lscape}
\usepackage{dsfont}
\usepackage{amsmath}
\usepackage{amsfonts}
\usepackage{tikz}
\usepackage{verbatim}
\usepackage{accents} 
\usepackage{todonotes}
\usepackage{mathtools}
\usepackage{bbm}

\usepackage[labelfont=bf, justification = justified, singlelinecheck=false]{caption}

\newcommand{\BlackBoxes}{\global\overfullrule5pt}

\BlackBoxes

\newcommand{\R}{\mathbb{R}} 
\newcommand{\dx}{\mathrm{d}}
\newcommand{\PP}{\mathbb{P}}

\newcommand{\E}{\mathbb{E}}

\newcommand{\st}{\text{s.t. }}

\newcommand{\Xpi}[2]{X_{#2}^{#1,\pi^{#1}}}

\newcommand{\amax}{\alpha_{\text{max}}}
\newcommand{\dmax}{\delta_{\text{max}}}

\mathtoolsset{showonlyrefs}

\newtheorem{theorem}{Theorem}

\newtheorem{lemma}[theorem]{Lemma}
\newtheorem{proposition}[theorem]{Proposition}
\theoremstyle{definition}
\newtheorem{example}[theorem]{Example}
\newtheorem{remark}[theorem]{Remark}
\newtheorem{definition}[theorem]{Definition}

\numberwithin{equation}{section} \numberwithin{theorem}{section}
\def\0{\kern0pt\-\nobreak\hskip0pt\relax}

\makeatletter
\AtBeginDocument{ \def\@serieslogo{ \vbox to\headheight{ \parindent\z@ \fontsize{6}{7\p@}\selectfont
\vss}}}

\def\makeoverbar#1#2#3#4#5#6#7{ \setbox0=\hbox{$\m@th#2\mkern#5mu{{}#3{}}\mkern#6mu$} \setbox1=\null \dimen@=#4\fontdimen8#13 \dimen@=3.5\dimen@
\advance\dimen@ by \ht0 \dimen@=-#7\dimen@ \advance\dimen@ by \wd0
\ht1=\ht0 \dp1=\dp0 \wd1=\dimen@
\dimen@=\fontdimen8#13 \fontdimen8#13=#4\fontdimen8#13
\rlap{\hbox to \wd0{$\m@th\hss#2{\overline{\box1}}\mkern#5mu$}}
\fontdimen8#13=\dimen@}
\def\mylabel#1#2{{\def\@currentlabel{#2}\label{#1}}}
\makeatother
\begin{document}
\title[Nash equilibria for relative investors with (non)linear price impact]{Nash equilibria for relative investors with (non)linear price impact}
\author[N. \smash{B\"auerle}]{Nicole B\"auerle}
\address[N. B\"auerle]{Department of Mathematics,
Karlsruhe Institute of Technology (KIT), D-76128 Karlsruhe, Germany, ORCID 0000-0003-0077-3444}

\email{\href{mailto:nicole.baeuerle@kit.edu}{nicole.baeuerle@kit.edu}}

\author[T. \smash{G\"oll}]{Tamara G\"oll}
\address[T. G\"oll]{Department of Mathematics,
Karlsruhe Institute of Technology (KIT), D-76128 Karls\-ruhe, Germany, ORCID 0009-0007-8341-723X}

\email{\href{mailto:tamara.goell@kit.edu} {tamara.goell@kit.edu}}


\begin{abstract}
We consider the strategic interaction of $n$ investors who are able to influence a stock price process and at the same time measure their utilities relative to the other investors. Our main aim is to find Nash equilibrium investment strategies in this setting in a financial market driven by a Brownian motion and investigate the influence the price impact has on the equilibrium. We consider both CRRA and CARA utility functions. Our findings show that the problem is well-posed as long as the price impact is at most linear. Moreover, numerical results reveal that the investors behave very aggressively when the price impact is close to a critical parameter. 
\end{abstract}
\maketitle


\makeatletter \providecommand\@dotsep{5} \makeatother



\vspace{0.5cm}
\begin{minipage}{14cm}
{\small
\begin{description}
\item[\rm \textsc{ Key words}]
{\small Portfolio optimization; Price Impact; Nash equilibrium; Relative investor}
\item[\rm \textsc{JEL subject classifications}] {C61, C73, G11}

\end{description}
}
\end{minipage}

\section{Introduction}

In this paper, we determine the optimal investment strategies of $n$ investors in a common financial market who interact strategically. The strategic interaction is caused by two different factors: a relative component inside the objective function of each investor and by the fact that the stock price dynamic is affected by the arithmetic mean of the $n$ agents' investments.

We contribute to two strands of literature. The first one is the literature on strategic interaction between agents. Strategic interaction in portfolio optimization problems has been motivated for example by \cite{brown2001careers} and \cite{kempf2008tournaments} through competition between agents. Since then, portfolio choice problems including strategic interaction between investors have been widely studied. The competitive feature is usually modeled through a relative performance metric. More specifically, either the additive relative performance metric, introduced by \cite{espinosa2010stochastic, espinosa2015optimal}, or the multiplicative performance metric, introduced by \cite{basak2014strategic}, are included into the utility function. \cite{basak2015competition} consider two agents in a continuous-time model which includes stocks following geometric Brownian motions. They use power utility functions and maximize the ratio of the two investors' wealth. \cite{espinosa2015optimal} also consider stocks driven by geometric Brownian motions and $n$ agents maximizing a weighted difference of their own wealth and the arithmetic mean of the other agents' wealth. Structurally similar objective functions including the arithmetic mean have been used by \cite{bauerle2021nash}. There, the unique Nash equilibrium for $n$ agents is derived in a very general financial market using the unique solution to some auxiliary classical portfolio optimization problem. \cite{lacker_zariphopoulou_n-agent_nash} consider the case of asset specialization for $n$ agents. They derive the unique constant Nash equilibrium using both the arithmetic mean under CARA utility and the geometric mean under CRRA utility. Later, their work has been extended by \cite{lacker2020many} to consumption-investment problems including relative concerns. In a similar asset specialization market with bounded market coefficients, \cite{fu2023mean2, fu2023mean} find a one-to-one correspondence between Nash equilibria and suitable systems of FBSDE's for agents applying power utilities to the multiplicative relative performance metric in order to find optimal investment (and consumption) strategies. \cite{dos2019forward, reis2020forward} use forward utilities of both CARA and CRRA type with and without consumption. More general financial markets (including e.g. stochastic volatility and incomplete information) were, for example, used in \cite{kraft2020dynamic}, \cite{fu2020mean} and \cite{hu2021n}. 

The second strand of literature focuses on (large) investors whose trades affect the price processes of certain assets. For an overview on reasons for the existence and methods to incorporate price impact, we refer to \cite{bouchaud2009price} and  \cite{webster2023handbook}. \cite{jarrow1992market, jarrow1994derivative}  considers a discrete time market model in which a single large trader affects the price of the risky asset. He finds conditions under which there are no arbitrage opportunities for small traders while the large trader is able to achieve riskless profit using some market manipulation strategy. \cite{almgren2001optimal} introduce a discrete-time financial market in which the price process of the risky stock is affected by the investment of a large investor. The impact is divided into \emph{temporary} and \emph{permanent} price impact. They minimize risk and transaction costs arising from the price impact simultaneously. Models including temporary and permanent price impact were also used by, among others, \cite{schied2017high, schied2017state, schied2019market}. In \cite{bertsimas1998optimal}, the problem of minimizing the expected cost of liquidating a block of shares over a fixed time interval is solved in a discrete time financial market. Here, the number of shares held by a large trader impacts the stock price process linearly.

\cite{cvitanic1996hedging} assume that the investment of a single large investor affects the interest rate of a riskless asset and the drift and volatility of stock price processes, which are modeled by It\^{o}-diffusions, simultaneously. They allow for general square integrable strategies and extend classical results of hedging contingent claims to their setting. A similar model including stocks paying dividends was used by \cite{cuoco1998optimal}. In their setting, the volatility of the stock prices does not depend on the large investors portfolio and they determine the optimal consumption strategy of the large investor.  \cite{bank2004hedging} use a more general continuous-time model for the stock prices, but only allow for constant portfolio processes. They prove necessary and sufficient conditions for the absence of arbitrage for both small and large investors. \cite{liu2005option} consider a Black-Scholes-type stock price dynamic where the investor's impact is modeled by a general price impact function integrated with respect to an It\^{o} process which models the investment of the large agent. After introducing their market model, they show how to price European options defined therein. \cite{eksi2017portfolio} also consider a Black-Scholes-type price process in which the drift is (possibly nonlinearly) affected by the large investor's trades and also contains a stochastic component which depends on the current market state. They maximize expected utility of the large investor under both complete and incomplete information. A problem of optimal liquidation in another Black-Scholes-type market is treated in \cite{he2005dynamic}. Here, the stock price depends linearly on the dynamics of the large investor's selling process. \cite{kraft2011large} maximize expected utility in a financial market similar to the one treated in this paper. They model the price process as a geometric Brownian motion by adding a multiple of the large trader's investment to the constant drift. A different approach to model price impact was used by \cite{bank2023optimal}. There, the large trader has additional information on the course of the future stock price and price impact is introduced as a penalty to exclude arbitrage opportunities.

The majority of literature considers the case of a single large trader. \cite{schoneborn2009liquidation}, however, consider a continuous time financial market where the price impact - both temporary and permanent - results from the investment of $n+1$  'strategic players'. Additionally, so-called market impact games, in which a finite number of large traders aims to minimize their liquidation/execution cost, have for example been considered by \cite{schied2017high, schied2019market, luo2019nash, fu2021mean}. Moreover, \cite{curatola2021price} considers two agents who interact strategically through their linear impact on the return of the risk free asset. Maximizing their terminal wealth under CRRA utility, he derives the unique constant pure-strategy Nash equilibrium. Risk-averse investors competing to maximize expected utility of terminal wealth have also been considered by \cite{schied2017state}.

In the following, we solve an $n$-agent portfolio problem with relative performance concerns where we allow the agents to jointly influence the asset dynamics, which is reasonable if $n$ is large, and which has not been done before.

This paper is organized as follows. In the next section, we introduce the linear price impact financial market. In Section \ref{sec:exp}, we explicitly solve the problem of maximizing expected exponential utility which results in the unique constant Nash equilibrium. The argument of the utility function consists of the difference of some agents' wealth and a weighted arithmetic mean of the other agents' wealth. We also examine the influence of the price impact parameter $\alpha$ to the Nash equilibrium and the stock price attained by inserting the arithmetic mean of the components of the Nash equilibrium. In Section \ref{sec:nonlinear}, we substitute the linear impact of the agents arithmetic mean on the stock price process by a nonlinear one. We prove that the problem of maximizing CARA utility is well-posed as long as the influence is sublinear and does not have an optimal solution if the influence is superlinear. In Section \ref{sec:power}, we assume that agents use CRRA utility functions (power and logarithmic utility) and insert the product of some agents wealth and a weighted geometric mean of the other agents' wealth into the expected utility criterion. Similar to the CARA case, we are able to explicitly determine the unique constant Nash equilibrium.

\section{Price impact market}\label{sec:price impact market}

Let $(\Omega, \mathcal{F}, (\mathcal{F}_t)_{t\in [0,T]}, \PP)$ be a filtered probability space and $T>0$ a finite time horizon. Moreover, let $W$ be a standard Brownian motion on $(\Omega, \mathcal{F}, (\mathcal{F}_t)_{t\in [0,T]}, \PP)$.

The underlying financial market consists of one riskless bond which will for simplicity be assumed to be identical to 1, and one risky asset (a stock). Note that it is straightforward to extend the results below to the case of $d>1$ stocks instead of just one (see \cite{goell2023expected} for analogous results including multiple stocks). However, to keep calculations simple, we only consider one stock.

The price process of the stock, denoted by $(S_t)_{t\in [0,T]}$, is the solution to the SDE
\begin{equation}
\dx S_t = S_t \left(\left(\mu + \alpha \bar{\pi}_t\right)\dx t +  \sigma \dx W_t\right),\, S(0)=1.\label{eq:price_impact_dynamic}
\end{equation}
Here, the drift $\mu>0$ and volatility $\sigma>0$ are assumed to be deterministic and constant in time. Our model describes a special case of the models considered by \cite{cvitanic1996hedging}, \cite{cuoco1998optimal} and \cite{kraft2011large}. Note that, instead of just one large investor, we consider the case of $n$ agents who collectively act like one large investor.

The expression $\bar{\pi}_t$ will describe the arithmetic mean of the investment of $n$ investors into the stock at time $t\in [0,T]$, i.e. 
\begin{equation}
\bar{\pi}_t \coloneqq \frac{1}{n}\sum_{i=1}^n \pi^i_t,
\end{equation}
where $\pi^i_t$ describes either the amount or the fraction of wealth agent $i$ invests into the stock at some time $t\in [0,T]$. 
The strategies $\pi^i$ of the $n$ investors are assumed to belong to the set $\mathcal{A}$ of $(\mathcal{F}_t)_{t \in [0,T]}$-progressively measurable, square-integrable processes, i.e. 
\begin{equation}
    \mathcal{A}\coloneqq \left\{\pi:\Omega \times [0,T] \to \R:\, \pi \text{ is } (\mathcal{F}_t)_{t\in [0,T]}\text{-progressively measurable, } \int_0^T \pi_t^2 \dx t < \infty \,\, \PP\text{-a.s.}\right\}.
\end{equation}
Further, let the initial capital of agent $i$ be given by $x_0^i$.

Finally, $\alpha \in \R$ is some constant that describes the impact of the investment of the $n$ investors into the stock. 

\begin{remark}
\begin{itemize}
    \item[a)] Some authors argue that $\alpha$ should take both positive and negative values due to the fact that (large) investors may have both positive and negative impact on stock returns (see e.g. \cite{cronqvist2008large}, \cite{curatola2019portfolio}). On the other hand, \cite{bank2004hedging} prove in a more general setting that stock prices need to be increasing in terms of some large investor's investment. Otherwise it would be possible to construct some 'In \& Out' arbitrage strategy. However, such arbitrage strategies arise due to the direct change in the share price in their model and are therefore not an issue in our case.  Moreover, since the optimization problems considered in this paper have finite optimal solutions, our model appears to be free of arbitrage. Hence, we allow for both positive and negative values for $\alpha$.
    \item[b)] Assuming that the drift of the risky stock depends linearly on the agents' investment makes the model mathematically tractable and can be seen as a first order approximation of nonlinear price impact (see \cite{kraft2011large}). However, empirical data suggests that price impact is concave in order size (see \cite{muhle2022stochastic} and references therein). Thus, we also consider the case of nonlinear price impact if investors use exponential utility functions (see Section \ref{sec:nonlinear}).
\end{itemize}

\end{remark}

\section{Optimization under CARA utility with linear price impact}\label{sec:exp}

At first, we assume that investors use exponential utility (CARA) functions to measure their preferences. Hence, define 
\begin{equation}
    U_i:\R \to \R, \, x\mapsto -\exp\Big(-\frac{1}{\delta_i}x\Big)
\end{equation}
for some parameter $\delta_i>0$, $i=1,\ldots,n.$ While using CARA utility functions, it is more convenient to consider the amount invested into the risky stock instead of the fraction of wealth or number of shares. Hence, we interpret $\pi^i_t$ as the amount of money agent $i$ invests into the risky stock at some $t\in [0,T],\, i=1,\ldots,n.$ Thus, the wealth process of agent $i$ is given by 
\begin{equation}
X_t^{i, \pi^i} = x_0^i + \int_0^t \pi^i_s \left(\left(\mu + \alpha \bar{\pi}_s\right)\dx s +  \sigma \dx W_s\right),\ t\in [0,T].
\end{equation}

In this paper, we want to examine the strategic interaction created by the price impact introduced earlier and a modification of the classical objective function used in expected utility maximization. Hence, we substitute the terminal wealth of a single investor inside the expected utility criterion by a relative quantity (a relative performance metric) which captures the fact that agent $i$ wants to maximize her terminal wealth while also considering her performance with respect to the other agents. Similar to \cite{bauerle2021nash} and Section 2 in \cite{lacker_zariphopoulou_n-agent_nash}, we use the difference of agent $i$'s terminal wealth and a weighted arithmetic mean of the other agents' terminal wealth. 
Hence, we insert $$X_T^{i,\pi^i} - \frac{\theta_i}{n}\sum_{j\neq i}X_T^{j,\pi^j}$$ into the argument of the utility function of investor $i$. The parameter $\theta_i\in [0,1]$ measures how much agent $i$ cares about her performance with respect to the other agents.

Our goal will therefore be to find all Nash equilibria to the multi-objective optimization problem
\begin{equation}\label{eq: price impact nash}
\begin{cases}
& \sup_{\pi^i \in \mathcal{A}} \E\left[-\exp\left( -\frac{1}{\delta_i}\left(X_T^{i,\pi^i} - \frac{\theta_i}{n}\sum_{j\neq i} X_T^{j,\pi^j}\right)\right) \right],\\
\st & X^{i,\pi^i}_T = x_0^i +  \displaystyle\int_0^T \pi^i_t\left(\left(\mu + \alpha \bar{\pi}_t\right) \dx t + \sigma \dx W_t \right),
\end{cases}
\end{equation}
$i=1,\ldots,n$. A Nash equilibrium for general objective functions $J_i$, $i=1,\ldots,n,$ is defined as follows. 
\begin{definition}\label{definition nash}
Let $J_i:\mathcal{A}^n \rightarrow \R$ be the objective function of agent $i$. A vector $\left(\pi^{1,*},\ldots,\pi^{n,*}\right)$ of strategies is called a \emph{Nash equilibrium}, if, for all admissible $\pi^i \in \mathcal{A}$ and $i=1,\ldots,n$, 
\begin{equation}
J_i(\pi^{1,*},\ldots,\pi^{i,*},\ldots,\pi^{n,*})\geq J_i(\pi^{1,*},\ldots,\pi^{i-1,*},\pi^i,\pi^{i+1,*},\ldots,\pi^{n,*}). \label{Nash condition}
\end{equation}
\end{definition}
I.e.\ deviating from $\pi^{i,*}$ does not increase agent $i$'s objective function. 

\subsection{Solution}
In order to solve the best response problem \eqref{eq: price impact nash}, we fix some investor $i$ and assume that the strategies $\pi^j$, $j\neq i$, of the other agents are given. Under these conditions we can rewrite the optimization problem \eqref{eq: price impact nash} into a classical portfolio optimization problem in a similar (but not identical) price impact market. Afterwards, the Nash equilibria can be determined using the solution to the classical problem. \medskip

Define the process $\big(Y_t^{i,\varphi^i}\big)_{t\in [0,T]}$ by 
\begin{equation}\label{eq: def Y_i}
Y_t^{i,\varphi^i} \coloneqq X^{i,\pi^i}_t - \frac{\theta_i}{n}\sum_{j\neq i} X^{j,\pi^j}_t,\, t\in [0,T], \,  i=1,\ldots,n,
\end{equation}
where we further define the strategy $\varphi^i$ by 
\begin{equation}\label{eq: def strategy Y_i}
\varphi^i_t \coloneqq \pi^i_t- \frac{\theta_i}{n}\sum_{j \neq i} \pi^j_t, \, t\in [0,T],\, i=1,\ldots,n,
\end{equation}
which is still square integrable and progressively measurable (i.e., $\varphi^i\in \mathcal{A}$). Then we can write $Y_T^{i,\varphi^i}$ as
\begin{align*}
Y_T^{i,\varphi^i} &= X^{i,\pi^i}_T - \frac{\theta_i}{n}\sum_{j\neq i} X^{j,\pi^j}_T \\
&\eqqcolon y_0^i + \int_0^T \varphi^i_t\left( \left(\widetilde{\mu}^{-i}_t + \frac{\alpha}{n}\varphi^i_t \right)\mathrm{d}t + \sigma \dx W_t\right),
\end{align*} 

where we introduced $y_0^i \coloneqq x_0^i - \frac{\theta_i}{n}\sum_{j\neq i} x_0^j$, $\bar{\pi}^{-i}_t \coloneqq \frac{1}{n}\sum_{j\neq i} \pi^j_t$ and $\widetilde{\mu}^{-i}_t \coloneqq \mu  + \alpha \frac{n+\theta_i}{n}\bar{\pi}^{-i}_t$.

Hence, in order to solve the best response problem associated to \eqref{eq: price impact nash}, we can equivalently solve the following single investor portfolio optimization problem due to the one-to-one relation between $\pi^i$ and $\varphi^i$
\begin{equation}\label{eq: price impact classic}
\begin{cases}
& \sup_{\varphi^i\in \mathcal{A}} \E\left[-\exp\left( -\frac{1}{\delta_i}Y_T^{i,\varphi^i}\right) \right],\\
\st & Y_T^{i,\varphi^i} = y_0^i + \int_0^T\varphi^i_t\left( \left(\widetilde{\mu}^{-i}_t + \frac{\alpha}{n}\varphi^i_t \right)\mathrm{d}t + \sigma\dx W_t\right),\, t\in [0,T],
\end{cases}
\end{equation}
in a financial market with corrected price impact. \medskip

Now assume that $\varphi^{i,*} = \varphi^{i,*}(\widetilde{\mu}^{-i})$ is an optimal solution to \eqref{eq: price impact classic} depending on the drift process $\widetilde{\mu}^{-i}$. Then the optimal solution to the best response problem for \eqref{eq: price impact nash} is uniquely determined by
\begin{equation}\label{eq: price impact system of equations}
\pi^i = \varphi^{i,*}(\widetilde{\mu}^{-i}) +  \frac{\theta_i}{n} \sum_{j\neq i} \pi^j,\, i=1,\ldots,n.
\end{equation}

Note that we can find a unique Nash equilibrium if and only if problem  \eqref{eq: price impact classic} and the fixed point problem for $\pi^i$, given in terms of the system of equations  \eqref{eq: price impact system of equations}, are uniquely solvable. 

Using the described technique, we are able to find the unique constant Nash equilibrium. Note that the restriction to constant Nash equilibria is necessary since otherwise we would not be able to solve the auxiliary problem \eqref{eq: price impact classic} explicitly.

As a first step, we solve the auxiliary problem \eqref{eq: price impact classic} for investor $i$ under the assumption that the strategies of the other investors are constant in time and deterministic.

\begin{lemma}\label{lemma:loesung_hilfsproblem}
Let $\theta_i\in [0,1]$ and $\delta_i>0$, $i=1,\ldots,n$. Moreover, assume that $n\sigma^2 - 2\delta_i\alpha>0$ for all $i=1,\ldots,n$. If, for some $i\in \{1,\ldots,n\}$, the strategies $\pi^j$, $j\neq i$, are constant in time and deterministic, the unique (up to sets of measure zero) optimal solution to \eqref{eq: price impact classic}  is given by 
\begin{equation}
    \varphi^{i,*}_t \equiv \frac{n\delta_i \widetilde{\mu}^{-i}}{n\sigma^2 - 2\delta_i \alpha},\, t\in [0,T].
\end{equation}
\end{lemma}

\begin{proof}

Since $\pi^j$, $j\neq i$, are constant, the drift $\widetilde{\mu}^{-i}$ is also constant. The dynamics of the wealth process $Y^{i,\varphi^i}$ are therefore given by 
\begin{align*}
\mathrm{d}Y_t^{i,\varphi^i} =  \varphi^i_t \left(\left(\widetilde{\mu}^{-i} + \frac{\alpha}{n} \varphi^i_t\right)\mathrm{d}t +  \sigma \mathrm{d}W_t\right),\, t\in [0,T].
\end{align*}

To derive the associated HJB equation used to solve the portfolio optimization problem, we define the value function
\begin{equation}
    v(t,y) := \sup_{\varphi^i\in \mathcal{A}} \E\Big[-\exp\Big(-\frac{1}{\delta_i}Y_T^{i,\varphi^i} \Big)\Big| Y_t^{i,\varphi^i}=y \Big],\ t\in [0,T], \, y\in \R.
\end{equation}

The maximum value in \eqref{eq: price impact classic} is thus given by $v(0,y_0^i).$  The Hamilton Jacobi Bellman (HJB) equation for this problem reads
\begin{equation}
    0 = v_t + \max_{a \in \R}\left\{v_y \widetilde{\mu}^{-i} a + \left(\frac{\alpha}{n}v_y + \frac{\sigma^2}{2}v_{yy} \right) a^2  \right\}\label{eq:hjb}
\end{equation}
for $y\in \R$, $t\in [0,T]$, with terminal condition $v(T,y)=-\exp(-\frac{1}{\delta_i}y).$ Note that we omitted the arguments of $v$ to keep notation simple. The maximum in \eqref{eq:hjb} is attained at 
\begin{equation}\label{eq:max_hjb}
    a^{*} = -\frac{n\widetilde{\mu}^{-i}v_y}{n\sigma^2 v_{yy} + 2\alpha v_y}.
\end{equation}
Inserting the maximum point into \eqref{eq:hjb} yields
\begin{equation}\label{eq:hjb_pde}
    0 = v_t - \frac{1}{2} \frac{n (\widetilde{\mu}^{-i})^2 v_y^2}{n\sigma^2 v_{yy} + 2\alpha v_y}.
\end{equation}

We use the ansatz $v(t,y) = -f(t)\exp(-\frac{1}{\delta_i}y)$ for some continuously differentiable function $f:[0,T]\to \R$ satisfying $f(T)=1$. Then \eqref{eq:hjb_pde} simplifies to the ODE
\begin{equation}
    f'(t) + \rho f(t) = 0,\, f(T)=1,
\end{equation}
where $\rho = -\frac{1}{2} \frac{n(\widetilde{\mu}^{-i})^2}{n\sigma^2 - 2 \alpha \delta_i}$. The unique solution to this ODE is given by $f(t) = e^{\rho(T-t)},\, t\in [0,T]$. Finally, 
$v(t,y) = - \exp(\rho(T-t)-\frac{1}{\delta_i}y)$, $t\in [0,T],\, y\in \R,$ solves the HJB equation. Inserting $v$ into \eqref{eq:max_hjb} yields 
\begin{equation}
    \varphi^{i,*} \equiv \frac{n\delta_i\widetilde{\mu}^{-i}}{n\sigma^2 - 2\delta_i \alpha}.
\end{equation}
A standard verification theorem (see for example \cite{bjork2004arbitrage}, pp.280-282, \cite{pham2009continuous}, \cite{fleming2006controlled} for similar versions) concludes our proof. In order to see that the optimal strategy is unique (up to sets of measures zero) note that optimal strategies have to satisfy the Bellman optimality principle (this has to be shown, but is standard, see \cite{pham2009continuous}, Thm. 3.3.1). Since we have already computed the value function, this necessarily implies that the optimal strategy is given by extremal points in the HJB equation (up to sets of measures zero). Since these maximum points are unique, the statement follows. 
\end{proof}

Lemma \ref{lemma:loesung_hilfsproblem} together with \eqref{eq: price impact system of equations} introduces a system of linear equations whose solutions constitute Nash equilibrium strategies. The next theorem displays the unique solution to this system and thus, the unique constant Nash equilibrium. In what follows, let $\hat{\theta} \coloneqq \sum_{j=1}^n \frac{\theta_j}{n+\theta_j}.$

\begin{theorem}\label{thm:solution}
Assume that  $n\sigma^2 - 2\delta_j\alpha > 0$ for all $j=1,\ldots,n.$ If $1 - \hat{\theta} \neq \sum_{j=1}^n \frac{n\alpha \delta_j}{(n+\theta_j)\left(n\sigma^2 - \delta_j\alpha\right)},$ the unique constant Nash equilibrium to  \eqref{eq: price impact nash} is given by 
\begin{align*}
\pi^{i,*} &= \frac{n}{n+\theta_i}\frac{n\delta_i\mu}{n\sigma^2 - \delta_i\alpha} + \left(\frac{\theta_i}{n+\theta_i} + \frac{n\alpha\delta_i}{(n+\theta_i)(n\sigma^2 - \delta_i\alpha)}\right)\cdot \frac{\sum_{j=1}^n \frac{n}{n+\theta_j}\frac{n\delta_j}{\left(n\sigma^2 - \delta_j \alpha\right)}\cdot \mu}{1 - \hat{\theta} - \sum_{j=1}^n \frac{n\alpha \delta_j}{(n+\theta_j)\left(n\sigma^2 - \delta_j\alpha\right)}},
\end{align*}
$i=1,\ldots,n.$  If $1 - \hat{\theta} = \sum_{j=1}^n \frac{n\alpha \delta_j}{(n+\theta_j)\left(n\sigma^2 - \delta_j\alpha\right)}$, there is no constant Nash equilibrium.
\end{theorem}

\begin{proof}
Using Lemma \ref{lemma:loesung_hilfsproblem}, the unique optimal solution to the auxiliary problem \eqref{eq: price impact classic} is given by 
\begin{equation}
    \varphi^{i,*} = \frac{n\delta_i \widetilde{\mu}^{-i}}{n\sigma^2 - 2\delta_i \alpha}.
\end{equation}
Note that this is obviously a constant  strategy. Moreover, we defined $\varphi^{i,*}=\pi^i - \frac{\theta_i}{n}\sum_{j\neq i} \pi^j$ and $\widetilde{\mu}^{-i} = \mu + \frac{n+\theta_i}{n^2}\alpha\sum_{j\neq i}\pi^j.$ Hence, we need to solve the following system of linear equations to determine the unique constant Nash equilibrium
\begin{equation}\label{eq:proof1}
    \pi^i - \frac{\theta_i}{n}\sum_{j\neq i}\pi^j = \frac{n\delta_i}{n\sigma^2 - 2\delta_i \alpha}\mu + \frac{\delta_i \alpha}{n\sigma^2 - 2\delta_i \alpha} \frac{n+\theta_i}{n}\sum_{j\neq i} \pi^j.
\end{equation}
Rearranging \eqref{eq:proof1} and adding $\pi^i$ in the sum yields
\begin{equation}\label{eq:proof2}
    \pi^i = \frac{n}{n+\theta_i}\frac{n\delta_i}{n\sigma^2 - \delta_i \alpha}\mu + \left(\frac{\theta_i}{n+\theta_i} + \frac{n\delta_i \alpha}{(n+\theta_i)(n\sigma^2 - \delta_i \alpha)} \right)\sum_{j=1}^n \pi^j.
\end{equation}
Summing over all $i\in \{1,\ldots,n\}$ on both sides then yields
\begin{equation}
    \sum_{j=1}^n \pi^j = \sum_{j=1}^n \frac{n}{n+\theta_j}\frac{n\delta_j}{n\sigma^2 - \delta_j \alpha} \mu + \bigg(\hat{\theta} + \sum_{j=1}^n \frac{n}{n+\theta_j}\frac{\delta_j \alpha}{n\sigma^2 - \delta_j \alpha} \bigg) \sum_{j=1}^n \pi^j. 
\end{equation}
Solving for $\sum_{j=1}^n \pi^j$ (which is possible if and only if $\sum_{j=1}^n \frac{n\alpha\delta_j}{(n+\theta_j)(n\sigma^2 - \delta_j\alpha)} \neq 1-\hat{\theta}$) yields 
\begin{equation}\label{eq:proof3}
    \sum_{j=1}^n \pi^j = \frac{\sum_{j=1}^n \frac{n}{n+\theta_j}\frac{n\delta_j}{\left(n\sigma^2 - \delta_j \alpha\right)}\cdot \mu}{1 - \hat{\theta} - \sum_{j=1}^n \frac{n\alpha \delta_j}{(n+\theta_j)\left(n\sigma^2 - \delta_j\alpha\right)}}.
\end{equation}
Finally, we can insert \eqref{eq:proof3} into \eqref{eq:proof2} to obtain the claimed representation of $\pi^{i,*}$ which concludes our proof.
\end{proof}

\begin{remark}\label{remark:special_exp}
Theorem \ref{thm:solution} contains the two special cases $\alpha = 0$ (no price impact) and $\theta_i=0$ for all $i=1,\ldots,n$ (no relative concerns in the objective function). For $\alpha=0$ we obtain
\begin{equation}
   \pi^{i,*} = \bigg(\frac{n\delta_i}{n+\theta_i} + \frac{\theta_i}{(1-\hat{\theta})(n+\theta_i)}\sum_{j=1}^n \frac{n\delta_j}{n+\theta_j} \bigg)\cdot \frac{\mu}{\sigma^2} >0
\end{equation}
for $i=1,\ldots,n$ which coincides (as expected) with the Nash equilibrium in \cite{bauerle2021nash} (Remark 4.1).  If $\theta_i = 0$ for all $i=1,\ldots,n,$ we deduce 
\begin{equation}
    \pi^{i,*} = \frac{n\delta_i\mu }{n\sigma^2 - \delta_i \alpha} + \frac{
    \alpha\delta_i}{n\sigma^2 - \delta_i \alpha} \cdot  \frac{n\sum_{j=1}^n \frac{\delta_j}{n\sigma^2 - \delta_j \alpha}}{1-\alpha \sum_{j=1}^n \frac{\delta_j}{n\sigma^2 - \delta_j \alpha}}\cdot \mu,
\end{equation}
$i=1,\ldots,n.$
\end{remark}

\subsection{Influence of the parameter $\alpha$}\label{subsec:alpha}
We consider two different features of our solution that are affected by the choice of the price impact parameter $\alpha$. Throughout this subsection, we assume that $\alpha$ satisfies the conditions of Theorem \ref{thm:solution}, i.e. $\alpha< \frac{n\sigma^2}{2\delta_{\text{max}}} \eqqcolon \alpha_{\text{max}}$, where  $\delta_{\text{max}} \coloneqq \max\{\delta_1,\ldots,\delta_n\}$, and $$\hat s(\alpha)\coloneqq \sum_{j=1}^n \frac{n\alpha\delta_j}{(n+\theta_j)(n\sigma^2 - \delta_j \alpha)}+\hat{\theta}\neq 1.$$ 
Indeed, it is possible to show that there exists a unique $\alpha_0\in (0,\alpha_{\text{max}})$ such that $\hat s(\alpha_0)=1.$ This can be seen as follows: First $\alpha\mapsto\hat s(\alpha)$ is strictly increasing and continuous on $(-\infty,\alpha_{\text{max}}].$ Further, we have $\hat s(0)=\hat \theta <1$ and $\hat s( \alpha_{\text{max}})>1.$ Thus, the intermediate value theorem implies the statement. We have to exclude this $\alpha_0$ from our considerations. The specific value of $\alpha_0$ does not depend on the type of the agent. It is the same for all investors.

First, we consider the impact of the choice of $\alpha$ on the optimal strategy of agent $i$, i.e. the $i$-th entry $\pi^{i,*}$ of the Nash equilibrium. It can be easily shown that $\pi^{i,*}>0$, $i=1,\ldots,n,$ if and only if $\alpha<\alpha_0$. Moreover, we can compute the derivative of $\pi^{i,*}$ with respect to $\alpha$ and deduce that it is strictly positive on $(-\infty,\alpha_{\text{max}})\setminus\{\alpha_0\}.$ Note, however, that $ \pi^{i,*}$ is only piecewise increasing on $(-\infty,\alpha_0)$ and $(\alpha_0, \alpha_{\text{max}})$ due to the discontinuity located at $\alpha_0$.

The second property of $\alpha$ we want to consider is the influence on the equilibrium stock price $(S_t^*)_{t\in [0,T]}$ that is obtained by inserting the Nash equilibrium from Theorem \ref{thm:solution} into the stock price dynamic. At first, it is not clear whether $S_t^*$ is smaller or larger than the stock price with drift $\mu$ and volatility $\sigma$ without the $n$ investors' impact. It obviously suffices to consider the drift of $\mathrm{d}S_t^*/S_t^*$ compared to $\mu$ since the volatility does not depend on the $n$ agents' investments. 

From the proof of Theorem \ref{thm:solution}, we know that the arithmetic mean of the components of the Nash equilibrium is given by
\begin{equation}
    \frac{1}{n}\sum_{j=1}^n \pi^{j,*} = \frac{(\hat s(\alpha)-\hat \theta)\cdot \mu/\alpha}{1 - \hat{s}(\alpha) }.
\end{equation}
Therefore, the drift of $S_t^*$ is equal to
\begin{align*}
    \mu + \frac{\alpha}{n}\sum_{j=1}^n \pi^{j,*} = \mu \cdot \frac{\hat s(\alpha)-\hat \theta}{1 - \hat{s}(\alpha) }.  
\end{align*}
Since the constant $\frac{\hat s(\alpha)-\hat \theta}{1 - \hat{s}(\alpha)}$ is strictly positive if and only if $\alpha \in (0,\alpha_0)$, we deduce that the drift of $S_t^*$ is larger (smaller) than $\mu$ if and only if $\alpha\in (0,\alpha_0)$ ($\alpha\in (-\infty,0)\cup (\alpha_0,\alpha_{\text{max}})$). Moreover, since we already saw that $\sum_{j=1}^n \pi^{j,*}$ is piecewise increasing in terms of $\alpha$, we infer that $S_t^*$ is also piecewise increasing in terms of $\alpha$ on $(-\infty,\alpha_0)$ and $(\alpha_0, \alpha_{\text{max}})$. More specifically, we obtain the following ordering 
    \begin{equation*}
        S_t^*(\alpha_3) < S_t^*(\alpha_1) < S_t^*(0) < S_t^*(\alpha_2)
    \end{equation*}
    for $\alpha_1 < 0,\, \alpha_2 \in (0,\alpha_0),\, \alpha_3 \in (\alpha_0,\amax)$. We refer to \cite{goell2023expected} for a more detailed discussion of the influence on the stock price.

Figure \ref{fig:pi_CARA} shows the behavior of $\pi^{1,*}$ from Theorem \ref{thm:solution} in terms of $\alpha$ for the two different risk aversion parameters $\delta_1=1$ and $\delta_1 = 4$. The vertical lines (dotted) show the discontinuity $\alpha_0$ for the different parameter choices. The gray horizontal line (dashed) marks the value zero while the orange and blue horizontal lines (dashed) display the optimal solution to the classical problem of maximizing expected terminal wealth under CARA utility without price impact and relative concerns given by $\delta_1 \mu \sigma^{-2}$.  There are two ways the agents may try to influence the stock price to their advantage. By buying the stock they may jointly increase the stock value and thus raise their utility or by jointly short-selling the stock and thus decrease its value. Our analysis shows that in case of a smaller price impact ($\alpha< \alpha_0$) the agents go for the first option and in case of a larger price impact ($\alpha> \alpha_0$) they go for the latter option.  Indeed, it turns out that there is a critical value $\alpha_0$ for the price impact where the  Nash equilibrium switches from positive to negative investment amounts. Around that value the agents trade very aggressively and try to outperform the others. Under an increasingly negative price impact, the investors engage less in the financial market which is not very surprising. 
If the price impact factor is further increased beyond $\alpha_0$ then the agents agree on investing less, because then it seems to be difficult to beat the performance of the others. 
Of course, this is only true under the exponential utility where short-selling is no problem. However, we will see later that for CRRA utilities a similar phenomenon occurs.

\begin{figure}[!tb] 
    \begin{center}
    \includegraphics[scale = 1]{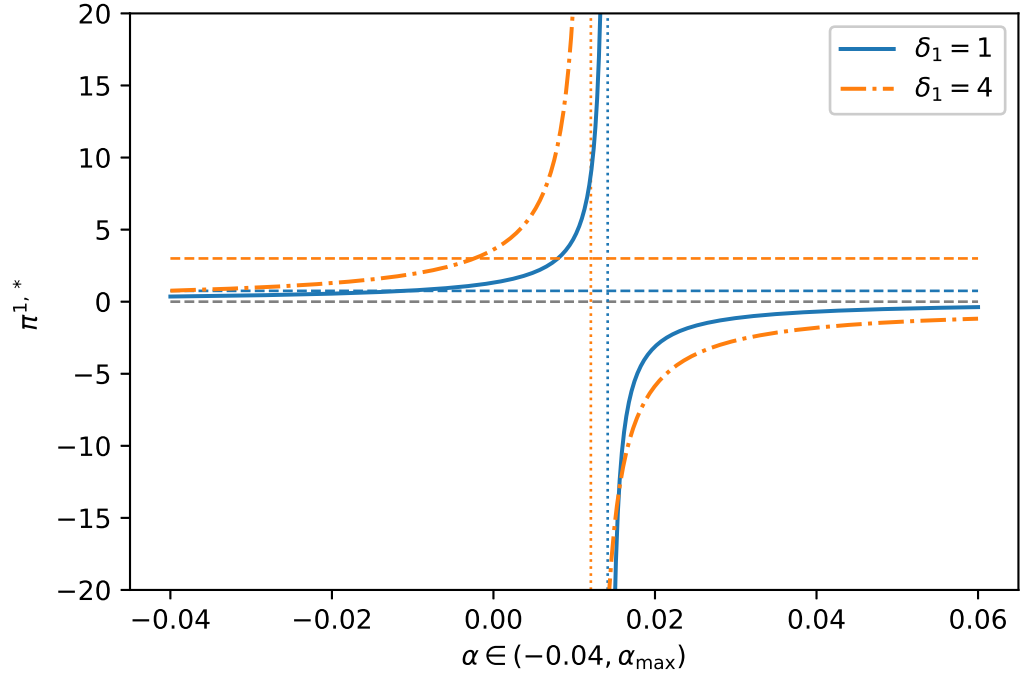}
    \end{center}
    \caption{Illustration of $\pi^{1,*}$ from Theorem \ref{thm:solution} in terms of $\alpha\in (-0.04,\alpha_{\text{max}})$ 
    for $n=12$, $\mu = 0.03$, $\sigma = 0.2$ and $\alpha_{\text{max}} = n \sigma^2 / 8=0.06$. $\theta_1 = 0.3,\, \delta_1\in \{1,4\}$ and the parameters $\theta_j$ and $ \delta_j$, $j\geq 2$ are increasing from $0$ to $1$ with step size $0.1$ and from $0.5$ to $2.7$ by step size $0.2$, respectively. The dashed blue and orange horizontal lines represent the optimal investment without price impact, given by $\delta_1 \mu\sigma^{-2}.$}
    \label{fig:pi_CARA}
\end{figure}

\section{Optimization under CARA utility with nonlinear price impact }\label{sec:nonlinear}
At the beginning of Section \ref{sec:price impact market}, we assumed that the price impact of the $n$ investors in our financial market is given as a linear function in terms of the arithmetic mean of the $n$ investors' strategies. While the use of the arithmetic mean seems intuitive and reasonable since we assumed that investors are 'small', one could ask whether using a different function than a linear one would lead to a different optimization problem and hence also a different Nash equilibrium. 

In Theorem \ref{thm:solution}, we were able to find an explicit solution to the associated multi-objective portfolio optimization problem using exponential utility (if the parameters are chosen accordingly). The proof highly relies on the linearity of the price impact, so we will not be able to give an explicit solution to the resulting optimization problem in general. However, we will discuss that using a function $g$ that grows superlinearly yields a problem that does not have a finite optimal solution while a function $g$ that grows sublinearly yields a finite optimal solution. If $g$ is a linear function, it depends on the parameter choices whether or not there exists a finite optimal solution (cf. Theorem \ref{thm:solution}). Since, in the linear case, the optimally invested amount is close to zero for decreasing price impact (i.e. if $\alpha<0$, see Theorem \ref{thm:solution} and Figure \ref{fig:pi_CARA}) we only consider price impact which is increasing in order size.

More explicitly, the price impact will now be modeled by some {\em strictly increasing and continuous function } $g:\R\to\R$ with $g(0)=0$. Therefore, the stock price process will be given as the solution to the SDE
\begin{equation}
\dx S_t = S_t\left(\left(\mu + g \left(\bar{\pi}_t\right)\right)\dx t + \sigma \dx W_t\right),\, S_0 = 1, 
\end{equation}
which is, of course, still just a stochastic exponential. 

As before, we have to restrict ourselves to constant Nash equilibria. Therefore, from the perspective of investor $i$, we can rewrite the expression $g(\bar{\pi}_t)$ in the previous SDE as follows
\begin{equation}
g(\bar{\pi}_t) = g\bigg(\frac{1}{n}\sum_{j=1}^n \pi^j_t\bigg) = g\bigg(\frac{1}{n}\pi^i_t + \frac{1}{n}\sum_{j\neq i} \pi^j\bigg) \eqqcolon \widetilde{g}(\pi^i_t),
\end{equation}
where $\widetilde{g}(p) \coloneqq g\left(\frac{p}{n}+ \frac{1}{n}\sum_{j\neq i} \pi^j\right),\, p\in \R$. Of course, we assumed that the strategies $\pi^j$, $j\neq i$, of the other investors are fixed, and constant. It also follows that $\widetilde{g}$ is still strictly increasing and satisfies $\widetilde{g}\left(-\sum_{j\neq i} \pi^j\right)=0$. 

Again, strategies $\pi^i$ are restricted to the set $\mathcal{A}$ of admissible strategies.

In the following, we will prove that 
\begin{equation}\label{eq:problem_g}
\begin{cases}
& \sup_{\pi^i \in \mathcal{A}} \E\left[-\exp\left( -\frac{1}{\delta_i}\left(X_T^{i,\pi^i} - \frac{\theta_i}{n}\sum_{j\neq i} X_T^{j,\pi^j}\right)\right) \right],\\
\st & X^{i,\pi^i}_T = x_0^i +  \displaystyle\int_0^T \pi^i_t\left(\left(\mu + g(\bar{\pi}_t)\right) \dx t + \sigma \dx W_t \right),
\end{cases}
\end{equation}
has an optimal solution if $g$ grows sublinearly and there exists no optimal strategy if $g$ grows superlinearly. 

The following theorem summarizes the first assertion of this section, which treats the case that $g$ grows superlinearly. 

\begin{proposition}\label{prop:nonlinear_price_impact}
If $\lim_{x\to \pm \infty} \frac{g(x)}{x} = \infty$, \eqref{eq:problem_g} does not have an optimal solution.
\end{proposition}

\begin{proof}
In order to prove that \eqref{eq:problem_g} does not have an optimal solution, we will prove that, even if we only consider constant strategies for agent $i$, the optimal value is zero and the associated strategy is infinite. If $\pi^j$ is constant for all $j\in \{1,\ldots,n\}$, we obtain
\begin{align*}
    &X_T^{i,\pi^i} - \frac{\theta_i}{n} \sum_{j\neq i} X_T^{j,\pi^j}\\
    =& x_0^i - \frac{\theta_i}{n}\sum_{j\neq i} x_0^j + \Big(\pi^i - \frac{\theta_i}{n}\sum_{j\neq i} \pi^j \Big)(\mu + g(\bar{\pi})) T + \Big(\pi^i - \frac{\theta_i}{n}\sum_{j\neq i} \pi^j\Big) \sigma W_T\\
    \eqqcolon & y_0^i + \mu(\pi^i)T + \sigma(\pi^i)W_T.
\end{align*}
Hence, for fixed $\pi^j$, $j=1,\ldots,n$, the value of the objective function in \eqref{eq:problem_g} is given by 
\begin{align}
    &\E\Big[-\exp\Big(-\frac{1}{\delta_i}\Big(y_0^i + \mu(\pi^i)T + \sigma(\pi^i)W_T\Big) \Big) \Big]\\
    =& -\exp\Big(-\frac{1}{\delta_i}y_0^i\Big) \cdot \exp\Big(-\frac{1}{\delta_i}\Big(\mu(\pi^i) - \frac{\sigma(\pi^i)^2}{2\delta_i} \Big)T \Big).
\end{align}
Thus, maximizing the objective function of $\eqref{eq:problem_g}$ with respect to constant strategies $\pi^i$ is equi\-valent to maximizing $ \mu(\pi^i) - \frac{\sigma(\pi^i)^2}{2\delta_i}.$
Reinserting the definition of $\mu(\pi^i)$ and $\sigma(\pi^i)$ yields 
\begin{align*}
    &\mu(\pi^i) - \frac{\sigma(\pi^i)^2}{2\delta_i} \\
=& \pi^i g(\bar{\pi}) - \frac{\sigma^2}{2\delta_i}(\pi^i)^2 + \pi^i\Big(\mu + \frac{\sigma^2 \theta_i}{n\delta_i}\sum_{j\neq i}\pi^j \Big)-\frac{\theta_i}{n}g(\bar{\pi})\sum_{j\neq i} \pi^j - \frac{\theta_i}{n}\sum_{j\neq i}\pi^j \Big(\mu + \frac{\sigma^2 \theta_i}{2n\delta_i}\sum_{j\neq i} \pi^j \Big)
\end{align*}
which converges to $\infty$ if $\pi^i$ converges to $\pm \infty$ using the assumption that $g$ grows superlinearly. Therefore, 
\begin{align*}
0 \geq &\sup_{\pi^i\in \mathcal{A}}\E\Big[-\exp\Big(-\frac{1}{\delta_i}\Big(X_T^{i,\pi^i} - \frac{\theta_i}{n} \sum_{j\neq i} X_T^{j,\pi^j} \Big) \Big) \Big]\\
\geq & \sup_{\substack{\pi^i\in \mathcal{A}\\ \pi^i \text{ constant}}}\E\Big[-\exp\Big(-\frac{1}{\delta_i}\Big(X_T^{i,\pi^i} - \frac{\theta_i}{n} \sum_{j\neq i} X_T^{j,\pi^j} \Big) \Big) \Big] = 0. 
\end{align*}
Hence, the optimal value of \eqref{eq:problem_g} is zero, which implies that the argument inside the exponential function needs to be infinite. Hence, the problem does not have an optimal solution.
\end{proof}

As a result, we cannot hope for a Nash equilibrium in this case.
Now we can consider the case of sublinear growth of $g$. Hence, we assume that 
$$\lim_{x\to \pm\infty} \frac{g(x)}{x} = 0.$$ 
Then we can prove that there  exists an optimal strategy for \eqref{eq:problem_g}. In order to do so, let $a^*$ be a maximum point of
\begin{equation}\label{eq:optprobsublinear}
     a \mapsto \Big(a- \frac{\theta_i}{n}\sum_{j\neq i} \pi^j \Big)\Big(\mu+\widetilde{g}(a)\Big) -\frac{\sigma^2}{2\delta_i} \Big(a- \frac{\theta_i}{n}\sum_{j\neq i} \pi^j \Big)^2.
\end{equation}
Due to our assumption on $g$, a maximum point $a^*$ exists and is finite. Then we obtain the following result.

\begin{proposition}\label{prop:sublinear}
If $\lim_{x\to \pm \infty} \frac{g(x)}{x} = 0$, an optimal strategy for \eqref{eq:problem_g} is given by $\pi_t^i\equiv a^*$, where $a^*$ is the maximum point from \eqref{eq:optprobsublinear}.
\end{proposition}

\begin{proof}
For  the moment, we restrict to bounded strategies $(\pi_t^i),$ i.e.\ there exists a constant $K>0$ such that $|\pi_t^i|\le K$ for all $t\in[0,T].$
 For constants $\pi^j$, we obtain
\begin{align*}
    &-\frac{1}{\delta_i}\Big(X_T^{i,\pi^i} - \frac{\theta_i}{n} \sum_{j\neq i} X_T^{j,\pi^j}\Big)\\
     =& -\frac{1}{\delta_i}\Big(x_0^i - \frac{\theta_i}{n}\sum_{j\neq i} x_0^j\Big) -\frac{1}{\delta_i}\Bigg( \int_0^T \Big(\pi^i_t - \frac{\theta_i}{n}\sum_{j\neq i} \pi^j \Big)(\mu + g(\bar{\pi}_t)) \dx t + \sigma \int_0^T \Big(\pi^i_t - \frac{\theta_i}{n}\sum_{j\neq i} \pi^j\Big) \dx W_t\Bigg)\\
     & - \frac{\sigma^2}{2\delta_i^2} \int_0^T \Big(\pi^i_t - \frac{\theta_i}{n}\sum_{j\neq i} \pi^j\Big)^2 \dx t+ \frac{\sigma^2}{2\delta_i^2} \int_0^T \Big(\pi^i_t - \frac{\theta_i}{n}\sum_{j\neq i} \pi^j\Big)^2 \dx t.
\end{align*}
Now define a new probability measure $\mathbb{Q}$ by
$$ \frac{\dx \mathbb{Q} }{\dx \mathbb{P}} = \exp\Bigg(  - \frac{\sigma^2}{2\delta_i^2} \int_0^T \Big(\pi^i_t - \frac{\theta_i}{n}\sum_{j\neq i} \pi^j\Big)^2 \dx t -  \frac{\sigma}{\delta_i} \int_0^T \Big(\pi^i_t - \frac{\theta_i}{n}\sum_{j\neq i} \pi^j\Big) \dx W_t\Bigg)
.$$
Note that this expression is a density since $\pi_t^i$ is bounded. Thus, we can write the (negative) target function of \eqref{eq:problem_g} as
\begin{align*}
    & \E\Bigg[\exp\Bigg( -\frac{1}{\delta_i}\Bigg(X_T^{i,\pi^i} - \frac{\theta_i}{n}\sum_{j\neq i} X_T^{j,\pi^j}\Bigg)\Bigg) \Bigg]
     = \exp\Bigg(-\frac{1}{\delta_i}\Big(x_0^i - \frac{\theta_i}{n}\sum_{j\neq i} x_0^j\Big)\Bigg) \E_{\mathbb{Q}}[Y^{\pi^i}],
\end{align*}
where
$$ Y^{\pi^i}:= \exp\Bigg( -\frac{1}{\delta_i}\Bigg( \int_0^T \Big(\pi^i_t - \frac{\theta_i}{n}\sum_{j\neq i} \pi^j \Big)(\mu + g(\bar{\pi}_t)) - \frac{\sigma^2}{2\delta_i}  \Big(\pi^i_t - \frac{\theta_i}{n}\sum_{j\neq i} \pi^j\Big)^2 \dx t\Bigg)\Bigg).$$

But now in order to minimize the expectation we can do this pointwise under the integral which leads to maximizing \eqref{eq:optprobsublinear}.  More precisely, note that $Y^{\pi^i}\le Y^{a^*}$ for all admissible $\pi^i$ and that $Y^{a^*}$ is deterministic. Hence, we obtain $\E_{\mathbb{Q}}[Y^{\pi^i}] \le Y^{a^*} = \E_{\mathbb{Q}^*}[Y^{a^*}]$. Since the maximizing point is not at the boundary, the assumption of bounded policies is no restriction. Thus, we have solved the problem.
\end{proof}

Whether or not a Nash equilibrium exists in this case depends on the precise choice of $g.$ Below, we provide an example of a function $g$ and parameter choices for which it is possible to determine a Nash equilibrium numerically.

\begin{remark}
    The structure of the function \eqref{eq:optprobsublinear} considered in the proof of Proposition \ref{prop:sublinear} implies that there exist at least one and at most two global maxima (see Remark 7.10 in \cite{goell2023expected} for a more detailed discussion). 
\end{remark}

\begin{example}
Let us provide a short numerical example in which there exists a unique constant Nash equilibrium under sublinear price impact. We consider two investors ($n=2$) and choose 
\begin{equation*}
    g(x) = \begin{cases}
        -\alpha (-x)^\gamma,\, & x<0,\\
        \alpha x^\gamma,\, & x\geq 0,
    \end{cases}
\end{equation*}
for some $\alpha>0$ and $\gamma \in (0,1]$. Note that $g$ satisfies the assumptions posed in the beginning of this section and grows sublinearly if $\gamma<1$. We included the case $\gamma=1$ for comparison to the linear case. \\
Let $\mu = 0.03,\, \sigma = 0.2,\, \delta_1 = 1,\, \delta_2=2,\, \theta_1 = 0.5,\, \theta_2 = 0.7$, and $\alpha = 0.01$. For the specific choice of parameters, we can determine the unique constant Nash equilibrium numerically by maximizing the function from \eqref{eq:optprobsublinear} for $i=1,2$ and solving the fixed point problem afterwards. The results are summarized in Figure \ref{fig:sublinear}. We included the Nash equilibrium in the case of linear price impact ($\gamma = 1$) for comparison (dashed horizontal lines). \\
Figure \ref{fig:sublinear} displays the behavior of the components $\pi^{i,*},\, i=1,2,$ of the constant Nash equilibrium in terms of the exponent $\gamma\in (0,1]$ of the price impact function $g$. The strategies $\pi^{i,*},\, i=1,2,$ in the Nash equilibrium are monotonically increasing in $\gamma$ and bounded from above by the strategies under linear price impact obtained in Theorem \ref{thm:solution}. Thus, the closer $g$ is to a linear function, the greater the resulting investment into the stock. The parameter choices for the two investors imply that agent 1 is more risk averse than agent 2 (note that a large choice of $\delta_i$ and $\theta_i$ can be associated to a more risk seeking investor, see, for example, Remark 3.5 in \cite{bauerle2021nash}). Thus, it does not come as a surprise that $\pi^{1,*}<\pi^{2,*}$ for all $\gamma \in (0,1]$.\\
It should be noted that some parameter choices do not yield a Nash equilibrium. A similar observation was made in the linear case (see Theorem \ref{thm:solution}). 

\begin{figure}
    \centering
    \includegraphics[scale = 1]{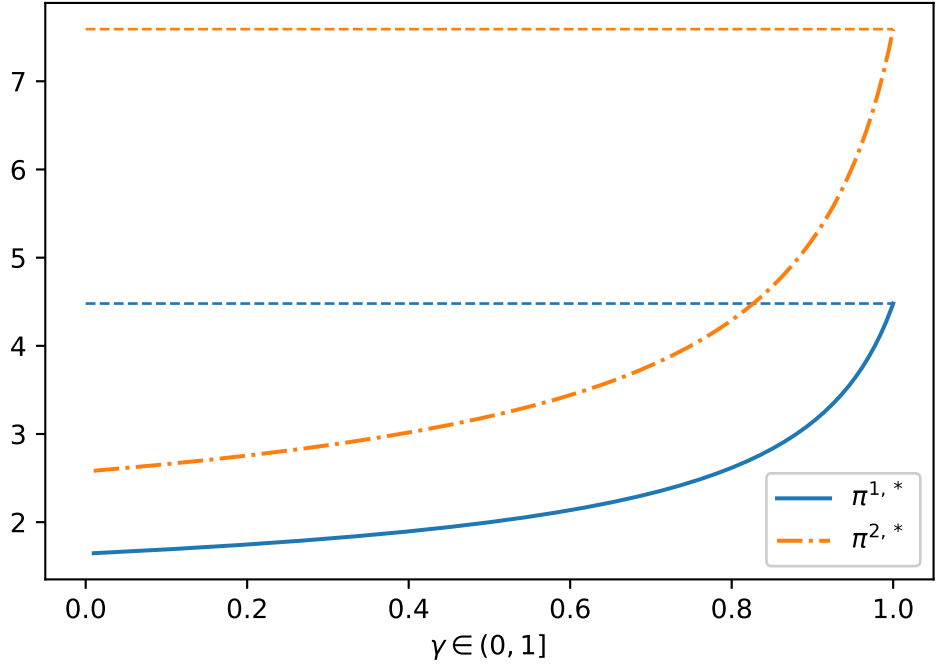}
    \caption{Illustration of the constant Nash equilibrium $(\pi^{1,*},\pi^{2,*})$ in terms of $\gamma\in (0,1]$ for the parameter choices $\mu = 0.03,\, \sigma = 0.2,\, \delta_1 = 1,\, \delta_2=2,\, \theta_1 = 0.5,\, \theta_2 = 0.7$, and $\alpha = 0.01$. The horizontal dashed lines represent the Nash equilibrium under linear price impact for comparison.}
    \label{fig:sublinear}
\end{figure}

\end{example}

\section{Optimization under CRRA utility with linear price impact}\label{sec:power}

In this section, we assume that agents use CRRA utility functions (power or logarithmic) to measure their preferences. Hence, we let 
\begin{equation}
U_i:(0,\infty) \rightarrow \R, \, x\mapsto \begin{cases}
\left(1-\frac{1}{\delta_i}\right)^{-1} x^{1-\frac{1}{\delta_i}},\, & \delta_i \neq 1,\\
\ln(x),\, & \delta_i = 1
\end{cases} 
\end{equation}
for some preference parameter $\delta_i > 0$, $i=1,\ldots,n$. By $\ln(\cdot)$ we denote the natural logarithm.\medskip

While using CRRA utility functions, it is mathematically more convenient to optimize the invested fraction of wealth instead of the amount or number of shares. Thus, throughout this subsection $\pi^i_t$, $i=1,\ldots,n,$ denotes the fraction of agent $i$'s wealth invested into the risky stock at some time $t\in [0,T]$.   However, we use the same SDE \eqref{eq:price_impact_dynamic} for the stock price as before. Thus, the interpretation of $\alpha$ in this model is different. The wealth process of agent $i$ is therefore given as the solution to the SDE 
\begin{equation}
\dx X_t^{i,\pi^i} = X_t^{i,\pi^i} \pi^i_t\left(\left(\mu + \alpha \bar{\pi}_t\right)\dx t +  \sigma \dx W_t\right),\, X_0^{i,\pi^i} = x_0^i.\label{eq:dynamic_X_price_impact_fraction}
\end{equation}

Similar to Section 3 in \cite{lacker_zariphopoulou_n-agent_nash}, we include the strategic interaction component into our problem by inserting the product of agent $i$'s and a weighted geometric mean of the other agents' terminal wealth into the expected utility criterion of the portfolio optimization problem. Therefore, the portfolio optimization problem of agent $i$ is given by 
\begin{equation}\label{eq:problem price impact power}
\begin{cases}
& \sup_{\pi^i\in \mathcal{A}} \, \E\Bigg[\frac{\delta_i}{\delta_i - 1} \Bigg(X_T^{i,\pi^i} \Big(\prod_{j\neq i} X_T^{j,\pi^j}\Big)^{-\frac{\theta_i}{n}}\Bigg)^{\frac{\delta_i-1}{\delta_i}} \Bigg],\\
\text{s.t. } & \dx X_t^{i,\pi^i} = X_t^{i,\pi^i} \pi^i_t\left(\left(\mu + \alpha \bar{\pi}_t\right)\dx t +  \sigma \dx W_t\right),\, X_0^{i,\pi^i} = x_0^i.
\end{cases}
\end{equation}

In order to find an explicit solution for the Nash equilibrium, we need to restrict ourselves to constant strategies. Since the reduction to some auxiliary problem containing only one instead of all $n$ agents is not possible in this setting, we need to directly solve the best response problem in order to determine the Nash equilibrium. Then the unique constant Nash equilibrium is given in the following theorem.

\begin{theorem}\label{thm:price impact power constant}
Assume that the following assumptions hold
\begin{enumerate}
\item $(n+\theta_i) \left(n\sigma^2 - \delta_i \alpha\right) - n\theta_i \delta_i \sigma^2\neq 0$ for all $i=1,\ldots,n$,
\item $n\sigma^2 - 2\delta_i\alpha>0$ for all $i=1,\ldots,n$,
\item $1 - \sum_{j=1}^n \frac{(n-\theta_j)\alpha\delta_j-n\theta_j(\delta_j-1)\sigma^2}{(n+\theta_j)(n\sigma^2-\alpha\delta_j) - n\theta_j\delta_j\sigma^2}\neq 0$.
\end{enumerate}
Then the unique (up to modifications) constant Nash equilibrium to \eqref{eq:problem price impact power} in terms of invested fractions is given by 
\begin{align}
    \pi^{i,*} = &\frac{n^2\delta_i \mu}{(n+\theta_i)(n\sigma^2 - \delta_i \alpha) - n\theta_i \delta_i \sigma^2} + \frac{(n-\theta_i)\alpha\delta_i - n\theta_i(\delta_i-1)\sigma^2}{(n+\theta_i)(n\sigma^2 - \delta_i \alpha)-n\theta_i\delta_i\sigma^2}\\
    &\cdot \Bigg(1 - \sum_{j=1}^n \frac{(n-\theta_j)\alpha\delta_j-n\theta_j(\delta_j-1)\sigma^2}{(n+\theta_j)(n\sigma^2-\alpha\delta_j) - n\theta_j\delta_j\sigma^2} \Bigg)^{-1} \sum_{j=1}^n \frac{n^2\delta_j\mu}{(n+\theta_j)(n\sigma^2 - \delta_j \alpha)-n\theta_j \delta_j\sigma^2}.
\end{align}
\end{theorem}

\begin{proof}
Let $i\in \{ 1,\ldots,n\}$ be arbitrary but fixed and assume that the other agents use constant strategies $\pi^j$, $j\neq i$, which will also be assumed to be arbitrary but fixed. Now define the stochastic process $(Y_t^{-i})_{t\in [0,T]}$ by $Y_t^{-i} = \prod_{j\neq i} X_t^{j,\pi^j}$, $t\in [0,T]$. \\

At first, we determine the dynamics of the process  $\Big((Y_t^{-i})^{-\frac{\theta_i}{n}}\Big)_{t\in [0,T]}$. To simplify our calculations, we first consider the logarithm of this process. We obtain 
\begin{equation} \label{eq: log Y}
\ln\Big((Y_t^{-i})^{-\frac{\theta_i}{n}}\Big) = -\frac{\theta_i}{n}\sum_{j\neq i} \ln\Big(X_t^{j,\pi^j}\Big)
\end{equation}
for $t\in [0,T]$. The It\^{o}-Doeblin formula implies
\begin{align*}
\dx \ln\Big(X_t^{j,\pi^j}\Big) &= \pi^j ((\mu + \alpha \bar{\pi}_t)\dx t + \sigma \dx W_t)-\frac{\sigma^2}{2}(\pi^j)^2  \dx t.
\end{align*} 
Hence, using \eqref{eq: log Y}, 
\begin{align*}
\dx \Big(\ln\Big((Y_t^{-i})^{-\frac{\theta_i}{n}}\Big)\Big) = - \frac{\theta_i}{n}\sum_{j\neq i} \pi^j ((\mu + \alpha \bar{\pi}_t)\dx t + \sigma \dx W_t)+\frac{\theta_i}{n}\frac{\sigma^2}{2}\sum_{j\neq i}(\pi^j)^2  \dx t.
\end{align*}
Using the It\^{o}-Doeblin formula a second time then yields 
\begin{align*}
&\dx \Big((Y_t^{-i})^{-\frac{\theta_i}{n}}\Big) = \dx \Big( \exp\Big( \ln \Big((Y_t^{-i})^{-\frac{\theta_i}{n}}\Big)\Big) \Big)\\
= &\,(Y_t^{-i})^{-\frac{\theta_i}{n}}\left(-\frac{\theta_i}{n}\sum_{j\neq i} \pi^j ((\mu + \alpha \bar{\pi}_t)\dx t + \sigma \dx W_t) + \frac{\sigma^2}{2}\frac{\theta_i}{n}\sum_{j\neq i} (\pi^j)^2 \dx t + \frac{\sigma^2}{2} \Big( \frac{\theta_i}{n}\Big)^2 \Big(\sum_{j\neq i} \pi^j \Big)^2 \dx t  \right).
\end{align*}

Hence, we can use partial integration to find the dynamics of the process associated to the argument of the utility function in \eqref{eq:problem price impact power}:
\begin{align*}
& \dx \Big(\Xpi{i}{t} (Y_t^{-i})^{-\frac{\theta_i}{n}} \Big) \\
= &\Xpi{i}{t} \left(Y_t^{-i}\right)^{-\frac{\theta_i}{n}} \Bigg(-\frac{\theta_i}{n}\sum_{j\neq i} \pi^j ((\mu + \alpha \bar{\pi}_t) \dx t + \sigma \dx W_t ) + \frac{\sigma^2}{2} \frac{\theta_i}{n}\sum_{j\neq i} (\pi^j)^2 \dx t + \frac{\sigma^2}{2}\Big(\frac{\theta_i}{n}\Big)^2 \Big(\sum_{j\neq i} \pi^j \Big)^2 \dx t\\
& + \pi^i_t ((\mu + \alpha \bar{\pi}_t) \dx t + \sigma\dx W_t) - \frac{\theta_i}{n}\sigma^2\pi^i _t\sum_{j\neq i}\pi^j \dx t\Bigg)\\
= &\Xpi{i}{t} \left(Y_t^{-i}\right)^{-\frac{\theta_i}{n}} \Bigg( \pi^i_t (\mu \dx t + \sigma \dx W_t) + \frac{\alpha}{n}(\pi^i_t)^2 \dx t + \Big(\frac{\alpha}{n} - \frac{\alpha \theta_i}{n^2}- \frac{\theta_i}{n}\sigma^2 \Big)\pi^i_t \sum_{j\neq i} \pi^j \dx t \\
& + \Big(\sum_{j\neq i}\pi^j \Big)^2 \frac{\theta_i}{n}\Big(\frac{\theta_i}{2n}\sigma^2 - \frac{\alpha}{n} \Big)\dx t - \frac{\theta_i}{n}\sum_{j\neq i} \pi^j (\mu \dx t + \sigma \dx W_t) + \frac{\theta_i}{2n}\sigma^2 \sum_{j\neq i}(\pi^j)^2 \dx t\Bigg),
\end{align*}
where we used the last step to separate the summands depending on $\pi^i$ from the ones that do not depend on $\pi^i$. Now a simple calculation yields that we can rewrite 
\begin{equation}
\Xpi{i}{t}\cdot \left(Y_t^{-i}\right)^{-\frac{\theta_i}{n}} = \widetilde{X}^{i,\pi^i}_t \cdot \left(\widetilde{Y}^{-i}_t\right)^{-\frac{\theta_i}{n}},\label{eq:rewrite_product}
\end{equation}
where the process $\widetilde{Y}^{-i}$ does not depend on $\pi^i$. More specifically, the dynamics of $\widetilde{X}^{i,\pi^i}$ and $\widetilde{Y}^{-i}$ are given by 
\begin{align*}
\dx \widetilde{X}^{i,\pi^i}_t &= \widetilde{X}^{i,\pi^i}_t \pi^i_t \Bigg(\Big(\mu + \frac{\alpha}{n}\pi^i_t + \frac{\alpha}{n} \Big(1-\frac{\theta_i}{n}\Big) \sum_{j\neq i} \pi^j\Big)\dx t + \sigma \dx W_t\Bigg),\\
\dx \widetilde{Y}^{-i}_t &= \widetilde{Y}^{-i}_t \Bigg(\sum_{j\neq i} \pi^j \Big(\Big(\mu + \frac{\alpha}{n}  \sum_{j\neq i} \pi^j
\Big)\dx t + \sigma \dx W_t\Big) + \frac{\sigma^2}{2}\Big( \sum_{j\neq i}\pi^j\Big)^2\dx t - \frac{\sigma^2}{2}\sum_{j\neq i}(\pi^j)^2 \dx t\Bigg)
\end{align*}
with $\widetilde{X}_0^{i,\pi^i} = x_0^i,\, \widetilde{Y}_0^{-i} = \prod_{j\neq i} x_0^j.$\medskip

The previously introduced processes $\widetilde{X}^{i,\pi^i}$ and $\widetilde{Y}^{-i}$ simplify the derivation of the HJB-equation in this setting. In order to derive an HJB-equation, we define the following value function ($t\in [0,T],\, x,y\in (0,\infty)$)
\begin{align*}
v(t,x,y) &\coloneqq \sup_{\pi^i\in \mathcal{A}} \E\left[\frac{\delta_i}{\delta_i-1} \left(\widetilde{X}^{i,\pi^i}_T \left(\widetilde{Y}^{-i}_T\right)^{-\frac{\theta_i}{n}}\right)^{\frac{\delta_i-1}{\delta_i}}\Bigg|\, \widetilde{X}^{i,\pi^i}_t =x,\, \widetilde{Y}^{-i}_t = y \right].
\end{align*}
We can derive an HJB equation using classical arguments (see e.g. \cite{pham2009continuous}, \cite{bjork2004arbitrage}, \cite{fleming2006controlled}) and obtain

\begin{align*}
0 =& v_t + y v_y \Bigg\{ \sum_{j\neq i}\pi^j \Big( \mu + \frac{\alpha}{n} \sum_{j\neq i}\pi^j\Big) + \frac{\sigma^2}{2}\Big(\sum_{j\neq i} \pi^j \Big)^2-\frac{\sigma^2}{2}\sum_{j=1}^n (\pi^j)^2\Bigg\} + \frac{\sigma^2}{2}y^2 v_{yy} \Big(\sum_{j\neq i} \pi^j \Big)^2\\
&+ \sup_{\pi^i\in \R}\Bigg\{x v_x \pi^i \mu +  \left( \frac{\alpha}{n}  \left(1-\frac{\theta_i}{n}\right)x v_x + \sigma^2xy v_{xy} \right)\pi^i \sum_{j\neq i} \pi^j  +\Big(\frac{\alpha}{n}xv_x + \frac{\sigma^2}{2} x^2 v_{xx}\Big) (\pi^i)^2 \Bigg\},
\end{align*}
where we omitted the arguments of $v$ and its derivatives for notational convenience. The supremum is attained at 
\begin{equation}
    \pi^{i,*} = -\frac{xv_x \mu + \Big(\frac{\alpha}{n}\Big(1-\frac{\theta_i}{n}\Big)xv_x + \sigma^2 xy v_{xy} \Big)\sum_{j\neq i} \pi^j}{2\Big(\frac{\alpha}{n}xv_x + \frac{\sigma^2}{2}x^2 v_{xx} \Big)},
\end{equation}
which reduces the HJB equation to the PDE 

\begin{align*}
0 =& v_t + y v_y \Bigg\{ \sum_{j\neq i}\pi^j \Big( \mu + \frac{\alpha}{n} \sum_{j\neq i}\pi^j\Big) + \frac{\sigma^2}{2}\Big(\sum_{j\neq i} \pi^j \Big)^2 -\frac{\sigma^2}{2}\sum_{j=1}^n (\pi^j)^2\Bigg\} + \frac{\sigma^2}{2}y^2 v_{yy} \Big(\sum_{j\neq i} \pi^j \Big)^2\\
&-\frac{\Big(xv_x \mu + \Big(\frac{\alpha}{n}\Big(1-\frac{\theta_i}{n}\Big)xv_x + \sigma^2 xy v_{xy} \Big)\sum_{j\neq i} \pi^j\Big)^2}{4\Big(\frac{\alpha}{n}xv_x + \frac{\sigma^2}{2}x^2 v_{xx} \Big)}
\end{align*}
with terminal condition 
\begin{equation}
    v(T,x,y) = \frac{\delta_i}{\delta_i-1}\Big(x y^{-\frac{\theta_i}{n}} \Big)^\frac{\delta_i-1}{\delta_i},\, x,y > 0.
\end{equation}

For the solution, we make the following ansatz for $v$
\begin{equation}
    v(t,x,y) = f(t) \frac{\delta_i}{\delta_i-1}\Big(x y^{-\frac{\theta_i}{n}} \Big)^\frac{\delta_i-1}{\delta_i}
\end{equation}
for some continuously differentiable function $f:[0,T]\to (0,\infty)$ with $f(T)=1$. Hence, inserting the ansatz for $v$ reduces the HJB equation to the ODE
\begin{align}\label{eq:ode_power}
    0 = f'(t) + \rho f(t)
\end{align}
with terminal condition $f(T)=1$, where we defined the constant
\begin{align*}
    \rho = &-\frac{\theta_i}{n}\frac{\delta_i-1}{\delta_i}\Bigg(\sum_{j\neq i} \pi^j \Big(\mu + \frac{\alpha}{n}\sum_{j\neq i} \pi^j \Big) + \frac{\sigma^2}{2}\Big(\sum_{j\neq i} \pi^j \Big)^2 -\frac{\sigma^2}{2}\sum_{j=1}^n (\pi^j)^2 \Bigg) \\
    &+ \frac{\sigma^2}{2}\frac{\theta_i}{n}\frac{\delta_i-1}{\delta_i}\Big(1+\frac{\theta_i}{n}\frac{\delta_i-1}{\delta_i}\Big)\Big(\sum_{j\neq i}\pi^j\Big)^2- \frac{\Big(n\delta_i \mu + \Big(\alpha\delta_i\Big(1-\frac{\theta_i}{n} \Big) - \sigma^2 \theta_i(\delta_i-1) \Big)\sum_{j\neq i} \pi^j\Big)^2}{4\Big( n\sigma^2 - 2\alpha \delta_i\Big)^2}.
\end{align*}

The unique solution to \eqref{eq:ode_power} is given by 
\begin{equation}
    f(t) = e^{\rho(T-t)},\, t\in [0,T].
\end{equation}

Inserting the solution $v$ of the HJB equation into the maximizer $\pi^{i,*}$ yields 

\begin{equation}\label{eq:sol_br_power}
    \pi^{i,*} = \frac{n\delta_i\mu + \Big(\alpha\delta_i\Big(1-\frac{\theta_i}{n}\Big) - \sigma^2 \theta_i(\delta_i-1) \Big)\sum_{j\neq i}\pi^j}{n\sigma^2 - 2\alpha \delta_i}.
\end{equation}

Application of a standard verification theorem (see for example \cite{pham2009continuous}, \cite{fleming2006controlled}, \cite{bjork2004arbitrage} for similar arguments) implies that $\pi^{i,*}$ is the unique solution to the best response problem. Moreover, since $\pi^j$ were assumed to be constant, $\pi^{i,*}$ is constant as well. To conclude the proof, we need to solve the system of linear equations defined by \eqref{eq:sol_br_power} for $i=1,\ldots,n$. By adding an appropriate multiple of $\pi^{i}$ on both sides and simplifying the equation, we obtain 
\begin{equation}\label{eq:pi_i_power}
    \pi^i = \frac{n\delta_i\mu }{(n+\theta_i)\Big(\sigma^2 - \frac{\delta_i\alpha}{n} \Big) - \sigma^2 \theta_i\delta_i}  + \frac{\alpha\delta_i\Big(1-\frac{\theta_i}{n} \Big) - \sigma^2\theta_i(\delta_i-1) }{(n+\theta_i)\Big(\sigma^2 - \frac{\delta_i\alpha}{n} \Big) - \sigma^2\theta_i\delta_i}\sum_{j=1}^n \pi^j.
\end{equation}
Summing over all $i\in \{1,\ldots,n\}$ on both sides and solving for $\sum_{j=1}^n \pi^j$ then yields 
\begin{equation}\label{eq:sum_power}
    \sum_{j=1}^n \pi^j = \Bigg(1- \sum_{j=1}^n \frac{\alpha\delta_j\Big(1-\frac{\theta_j}{n}\Big) - \sigma^2\theta_j (\delta_j-1)}{(n+\theta_j)\Big(\sigma^2 - \frac{\delta_j \alpha}{n}\Big) - \sigma^2 \theta_j\delta_j}  \Bigg)^{-1}\sum_{j=1}^n \frac{n\delta_j\mu}{(n+\theta_j)\Big(\sigma^2 - \frac{\delta_j \alpha}{n}\Big)-\sigma^2\theta_j \delta_j }
\end{equation}
Finally, inserting \eqref{eq:sum_power} into \eqref{eq:pi_i_power} yields the unique constant Nash equilibrium given by ($i=1,\ldots,n$)
\begin{align*}
    \pi^{i,*} = &\frac{n\delta_i\mu }{(n+\theta_i)\Big(\sigma^2 - \frac{\delta_i\alpha}{n} \Big) - \sigma^2 \theta_i\delta_i}  + \frac{\alpha\delta_i\Big(1-\frac{\theta_i}{n} \Big) - \sigma^2\theta_i(\delta_i-1) }{(n+\theta_i)\Big(\sigma^2 - \frac{\delta_i\alpha}{n} \Big) - \sigma^2\theta_i\delta_i}\\
    &\cdot \Bigg(1- \sum_{j=1}^n \frac{\alpha\delta_j\Big(1-\frac{\theta_j}{n}\Big) - \sigma^2\theta_j (\delta_j-1)}{(n+\theta_j)\Big(\sigma^2 - \frac{\delta_j \alpha}{n}\Big) - \sigma^2 \theta_j\delta_j}  \Bigg)^{-1}\sum_{j=1}^n \frac{n\delta_j\mu}{(n+\theta_j)\Big(\sigma^2 - \frac{\delta_j \alpha}{n}\Big)-\sigma^2\theta_j \delta_j}.
\end{align*}

\end{proof}

\begin{remark}\label{remark:special_pow}
Similar to Remark \ref{remark:special_exp}, Theorem \ref{thm:price impact power constant} contains the special cases $\alpha=0$ and $\theta_i = 0$, $i=1,\ldots,n$. For $\alpha=0$ (no price impact), we deduce
\begin{equation}
    \pi^{i,*}= \Bigg(\frac{n\delta_i}{n+\theta_i(1-\delta_i)} + \frac{\theta_i(1-\delta_i)}{n+\theta_i(1-\delta_i)}\cdot \frac{\sum_{j=1}^n\frac{n\delta_j}{n+\theta_j(1-\delta_j)}}{1-\sum_{j=1}^n \frac{\theta_j(1-\delta_j)}{n+\theta_j(1-\delta_j)}}\Bigg)\cdot \frac{\mu}{\sigma^2},
\end{equation}
$i=1,\ldots,n$. In the special case without relative concerns inside the objective function, we obtain
\begin{equation}
    \pi^{i,*}=\frac{n\delta_i\mu}{n\sigma^2 - \delta_i \alpha} + \frac{\alpha\delta_i}{n\sigma^2 - \delta_i \alpha}\cdot \frac{\sum_{j=1}^n \frac{n\delta_j \mu }{n\sigma^2 - \alpha\delta_j}}{1-\sum_{j=1}^n \frac{\alpha\delta_j}{n\sigma^2 - \alpha\delta_j}},
\end{equation}
$i=1,\ldots,n.$ A comparison with Remark \ref{remark:special_exp} shows that the Nash equilibria in the special case of $\theta_i = 0$ for all $i=1,\ldots,n$ are actually the same, although $\pi^{i,*}$ represents the invested amount for exponential and the invested fraction for power utility. 
\end{remark}

In Subsection \ref{subsec:alpha}, we analyzed the influence of the price impact parameter $\alpha$ on the entries of the constant Nash equilibrium in terms of invested amounts under exponential utility. Due to the more complicated structure of the constant Nash equilibrium in Theorem \ref{thm:price impact power constant}, we do not discuss the influence of $\alpha$ as detailed as in Subsection \ref{subsec:alpha}. However, Theorem \ref{thm:price impact power constant} enables us to explicitly compute the value of a component of the Nash equilibrium for specific parameter choices. The results are illustrated in  Figure \ref{fig:price_impact_power}.

\begin{figure}[!h]
\centering
\includegraphics[scale=1]{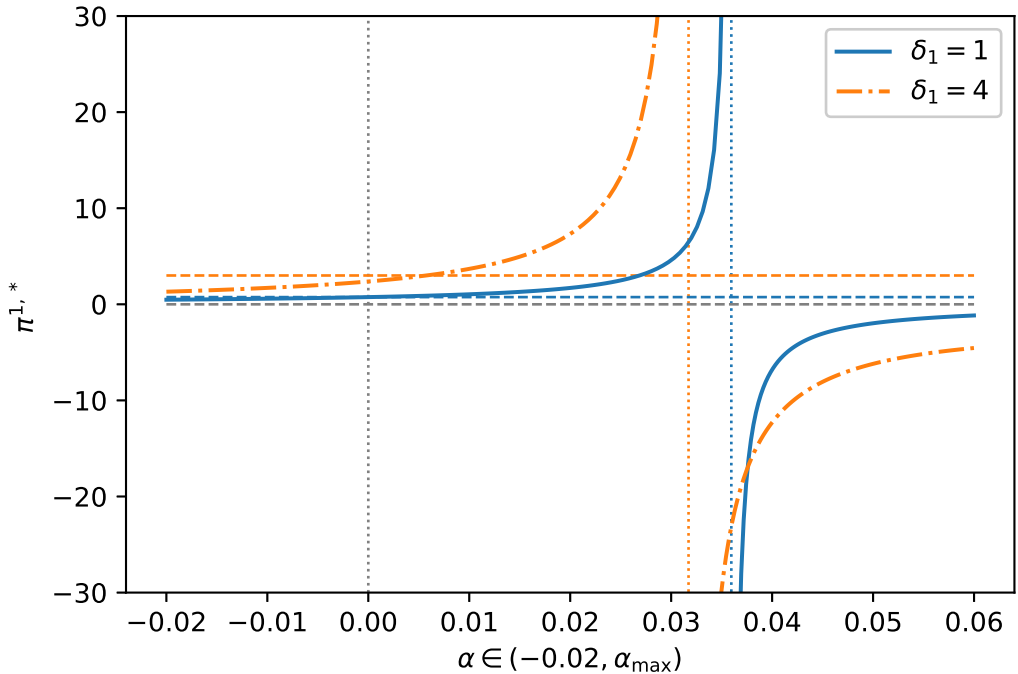}
\caption{Illustration of $\pi^{1,*}$ from Theorem \ref{thm:price impact power constant} in terms of $\alpha\in (-0.02,\alpha_{\text{max}})$ for $n=12$, $\mu = 0.03$, $\sigma = 0.2$, and $\alpha_{\text{max}} = n \sigma^2 / 8=0.06$. Further, $\theta_1 = 0.3,\, \delta_1\in \{1,4\}$ and the parameters $\theta_j$ and $ \delta_j$, $j\geq 2$ are increasing from $0$ to $1$ with step size $0.1$, and from $0.5$ to $2.7$ by step size $0.2$, respectively. The dashed blue and orange horizontal lines represent the optimal investment fraction without price impact, given by $\delta_1 \mu\sigma^{-2}.$}
\label{fig:price_impact_power}
\end{figure}

Figure \ref{fig:price_impact_power} displays the first component $\pi^{1,*}$ of the constant Nash equilibrium given in Theorem~\ref{thm:price impact power constant} in terms of $\alpha$ varying between $-0.02$ and $\alpha_\text{max}$ for the two different risk tolerance parameters $\delta_1 =1 $ and $\delta_1 = 4$. The expression $\alpha_\text{max}$ is defined analogously to Subsection \ref{subsec:alpha} as 
\begin{equation}
\amax = \frac{n\sigma^2}{2\dmax },
\end{equation}
where $\dmax = \max\{\delta_1,\ldots,\delta_n\}$. In the example displayed in Figure \ref{fig:price_impact_power}, we used $\dmax = 4$. The market parameters are chosen as $\mu = 0.03$ and $\sigma = 0.2$. Note that all considered parameter combinations satisfy the conditions of Theorem \ref{thm:price impact power constant}. Similar to Figure \ref{fig:pi_CARA}, we observe a discontinuity of $\pi^{1,*}$. In Subsection \ref{subsec:alpha}, we provided a detailed discussion of the existence of a unique point of discontinuity. Here, we only give a short explanation regarding the discontinuity. For the specific parameter choices used in the example, conditions a) and b) of Theorem \ref{thm:price impact power constant} are always satisfied. The discontinuity is due to condition c), i.e., for both parameter choices $\delta_1 \in \{1,4\}$, there exists a unique value $\alpha_0 \in (-\infty, \amax)$ such that the expression in condition~c) is zero. In the figure, the value $\alpha_0$ is highlighted by a vertical dotted line for each of the two parameter choices $\delta_1 \in \{1,4\}$. Moreover, the blue and orange horizontal dashed lines mark the Merton ratio $\delta_1 \mu \sigma^{-2}$, i.e., the unique optimally invested fraction in the associated classical problem ($\alpha = 0$, $\theta_1 = 0$), for the two different values used for $\delta_1$. Finally, we highlighted the value zero on both axes by a grey line. \medskip

Considering the behavior of $\pi^{1,*}$ in terms of $\alpha$, we notice that $\pi^{1,*}$ is strictly positive for $\alpha<\alpha_0$ and strictly negative for $\alpha> \alpha_0$. Moreover, we observe that for larger price impact (i.e., if the absolute value of $\alpha$ increases), the agents engage less in the financial market which is represented by a decrease in the absolute value of $\pi^{1,*}$. Overall, we notice a similar behavior of $\pi^{1,*}$ in terms of $\alpha$ as in the case of exponential utility which we considered in Subsection \ref{subsec:alpha}.

\begin{remark}
    In contrast to the discussion of nonlinear price impact in the CARA case (see Section \ref{sec:nonlinear}), we only consider linear price impact for CRRA utility. In the CARA case, relative concerns are included into the objective function linearly which simplifies the problem in comparison to the CRRA case, in which the relative concerns are included multiplicatively. It seems to us that the treatment of nonlinear price impact would be a lot more tedious in the CRRA model. Moreover, we expect that the gain of insight would be minimal as a similar threshold phenomenon should appear in the CRRA case. To understand this conjecture, consider the terminal wealth of agent $i$ in the CRRA model with nonlinear price impact
\begin{equation*}
    X_T^{i,\pi^i} = x_0^i \exp\left(\int_0^T \pi_t^i\big( (\mu + g(\bar{\pi}_t)) dt + \sigma dW_t \big) - \frac{1}{2}\int_0^T \sigma^2 (\pi_t^i)^2 dt\right).
\end{equation*}

It appears that, as long as $g$ grows at most linearly, the increase in the drift is balanced by the quadratic influence in the second integral. If, however, $g$ grows superlinearly, we expect the first integral in the exponential to be dominant, resulting in an unbounded best response problem and, thus, no Nash equilibrium in this case. 
\end{remark}

\section{Conclusion}
In this paper we derive Nash equilibria for agents with relative performance measures in financial markets with price impact. We show that as long as the price impact is not more than linear, the individual optimization problems are well-defined. Whereas without price impact, the agents would always invest a positive amount in the stock in our model, the situation changes dramatically when a price impact is present. Then there exists a critical number for the price impact variable where the Nash equilibrium changes from a situation where all investors try to increase the stock to a situation where they try to decrease the stock.\\

\textit{Acknowledgment:} The authors would like to thank Dirk Becherer and Johannes Muhle-Karbe for helpful discussions and hints to literature. Further, they are grateful to two anonymous referees for their suggestions which helped to improve the paper.

\textit{Statements and Declarations:} The authors have no relevant financial or non-financial interests to disclose. Data sharing is not applicable to this article as no datasets were generated or analyzed during the current study.

\bibliographystyle{amsplain}
\bibliography{literatur_main}

\providecommand{\bysame}{\leavevmode\hbox to3em{\hrulefill}\thinspace}
\providecommand{\MR}{\relax\ifhmode\unskip\space\fi MR }
\providecommand{\MRhref}[2]{%
  \href{http://www.ams.org/mathscinet-getitem?mr=#1}{#2}
}
\providecommand{\href}[2]{#2}
\begin{thebibliography}{10}

\bibitem{almgren2001optimal}
Robert Almgren and Neil Chriss, \emph{Optimal execution of portfolio
  transactions}, Journal of Risk \textbf{3} (2001), 5--40.

\bibitem{bank2004hedging}
Peter Bank and Dietmar Baum, \emph{Hedging and portfolio optimization in
  financial markets with a large trader}, Mathematical Finance: An
  International Journal of Mathematics, Statistics and Financial Economics
  \textbf{14} (2004), no.~1, 1--18.

\bibitem{bank2023optimal}
Peter Bank and Yan Dolinsky, \emph{Optimal investment with a noisy signal of
  future stock prices},
  \href{https://arxiv.org/abs/2302.10485}{arXiv:2302.10485\hspace*{-1mm}}
  (2023).

\bibitem{basak2014strategic}
Suleyman Basak and Dmitry Makarov, \emph{Strategic asset allocation in money
  management}, The Journal of finance \textbf{69} (2014), no.~1, 179--217.

\bibitem{basak2015competition}
\bysame, \emph{Competition among portfolio managers and asset specialization},
  Paris December 2014, Finance Meeting EUROFIDAI-AFFI Paper (2015).

\bibitem{bauerle2021nash}
Nicole B{\"a}uerle and Tamara G{\"o}ll, \emph{Nash equilibria for relative
  investors via no-arbitrage arguments}, Mathematical Methods of Operations
  Research (2022), 1--23.

\bibitem{bertsimas1998optimal}
Dimitris Bertsimas and Andrew~W. Lo, \emph{Optimal control of execution costs},
  Journal of Financial Markets \textbf{1} (1998), no.~1, 1--50.

\bibitem{bjork2004arbitrage}
Tomas Bj{\"o}rk, \emph{Arbitrage theory in continuous time}, 2. ed., Oxford
  University Press, 2004.

\bibitem{bouchaud2009price}
Jean-Philippe Bouchaud, \emph{Price impact}, arXiv:0903.2428 (2009).

\bibitem{brown2001careers}
Stephen~J. Brown, William~N. Goetzmann, and James Park, \emph{Careers and
  survival: Competition and risk in the hedge fund and {CTA} industry}, The
  Journal of Finance \textbf{56} (2001), no.~5, 1869--1886.

\bibitem{cronqvist2008large}
Henrik Cronqvist and R{\"u}diger Fahlenbrach, \emph{Large shareholders and
  corporate policies}, The Review of Financial Studies \textbf{22} (2009),
  no.~10, 3941--3976.

\bibitem{cuoco1998optimal}
Domenico Cuoco and Jak{\v{s}}a Cvitani{\'c}, \emph{Optimal consumption choices
  for a ‘large’investor}, Journal of Economic Dynamics and Control
  \textbf{22} (1998), no.~3, 401--436.

\bibitem{curatola2019portfolio}
Giuliano Curatola, \emph{Portfolio choice of large investors who interact
  strategically}, Available at SSRN 3404491 (2019).

\bibitem{curatola2021price}
\bysame, \emph{Price impact, strategic interaction and portfolio choice}, The
  North American Journal of Economics and Finance \textbf{59} (2022), 101594.

\bibitem{cvitanic1996hedging}
Jak{\v{s}}a Cvitani{\'c} and Jin Ma, \emph{Hedging options for a large investor
  and forward-backward sde's}, The Annals of Applied Probability \textbf{6}
  (1996), no.~2, 370--398.

\bibitem{dos2019forward}
Goncalo Dos~Reis and Vadim Platonov, \emph{Forward utilities and mean-field
  games under relative performance concerns}, From Particle Systems to Partial
  Differential Equations: International Conference, Particle Systems and PDEs
  VI, VII and VIII, 2017-2019 (Cham) (C{\'e}dric Bernardin, Fran{\c{c}}ois
  Golse, Patr{\'i}cia Gon{\c{c}}alves, Valeria Ricci, and Ana~Jacinta Soares,
  eds.), vol. 352, Springer, 2021, pp.~227--251.

\bibitem{reis2020forward}
\bysame, \emph{Forward utility and market adjustments in relative
  investment-consumption games of many players}, SIAM Journal on Financial
  Mathematics \textbf{13} (2022), no.~3, 844--876.

\bibitem{eksi2017portfolio}
Zehra Eksi and Hyejin Ku, \emph{Portfolio optimization for a large investor
  under partial information and price impact}, Mathematical Methods of
  Operations Research \textbf{86} (2017), 601--623.

\bibitem{espinosa2010stochastic}
Gilles-Edouard Espinosa, \emph{Stochastic control methods for optimal portfolio
  investment}, Ph.D. thesis, Ecole Polytechnique Paris, 2010.

\bibitem{espinosa2015optimal}
Gilles-Edouard Espinosa and Nizar Touzi, \emph{Optimal investment under
  relative performance concerns}, Mathematical Finance \textbf{25} (2015),
  no.~2, 221--257.

\bibitem{fleming2006controlled}
Wendell~H Fleming and Halil~Mete Soner, \emph{Controlled markov processes and
  viscosity solutions}, vol.~25, Springer Science \& Business Media, 2006.

\bibitem{fu2023mean2}
Guanxing Fu, \emph{Mean field portfolio games with consumption}, Mathematics
  and Financial Economics \textbf{17} (2023), no.~1, 79--99.

\bibitem{fu2021mean}
Guanxing Fu, Paulwin Graewe, Ulrich Horst, and Alexandre Popier, \emph{A mean
  field game of optimal portfolio liquidation}, Mathematics of Operations
  Research \textbf{46} (2021), no.~4, 1250--1281.

\bibitem{fu2020mean}
Guanxing Fu, Xizhi Su, and Chao Zhou, \emph{Mean field exponential utility
  game: A probabilistic approach}, arXiv:2006.07684 (2020).

\bibitem{fu2023mean}
Guanxing Fu and Chao Zhou, \emph{Mean field portfolio games}, Finance and
  Stochastics \textbf{27} (2023), no.~1, 189--231.

\bibitem{goell2023expected}
Tamara Göll, \emph{Expected utility maximzation for competitive agents}, Ph.D.
  thesis, Karlsruhe Institute of Technology, 2023,
  \href{https://doi.org/10.5445/IR/1000167954}{DOI:10.5445/IR/1000167954}.

\bibitem{he2005dynamic}
Hua He and Harry Mamaysky, \emph{Dynamic trading policies with price impact},
  Journal of Economic Dynamics and Control \textbf{29} (2005), no.~5, 891--930.

\bibitem{hu2021n}
Ruimeng Hu and Thaleia Zariphopoulou, \emph{$n$-player and mean-field games in
  {I}t\^{o}-diffusion markets with competitive or homophilous interaction},
  Stochastic Analysis, Filtering, and Stochastic Optimization, Springer, Cham,
  2022.

\bibitem{jarrow1992market}
Robert~A. Jarrow, \emph{Market manipulation, bubbles, corners, and short
  squeezes}, Journal of Financial and Quantitative Analysis \textbf{27} (1992),
  no.~3, 311--336.

\bibitem{jarrow1994derivative}
\bysame, \emph{Derivative security markets, market manipulation, and option
  pricing theory}, Journal of Financial and Quantitative Analysis \textbf{29}
  (1994), no.~2, 241--261.

\bibitem{kempf2008tournaments}
Alexander Kempf and Stefan Ruenzi, \emph{Tournaments in mutual-fund families},
  The Review of Financial Studies \textbf{21} (2008), no.~2, 1013--1036.

\bibitem{kraft2011large}
Holger Kraft and Christoph K{\"u}hn, \emph{Large traders and illiquid options:
  Hedging vs. manipulation}, Journal of Economic Dynamics and Control
  \textbf{35} (2011), no.~11, 1898--1915.

\bibitem{kraft2020dynamic}
Holger Kraft, Andr{\'e} Meyer-Wehmann, and Frank~Thomas Seifried, \emph{Dynamic
  asset allocation with relative wealth concerns in incomplete markets},
  Journal of Economic Dynamics and Control \textbf{113} (2020), 103857.

\bibitem{lacker2020many}
Daniel Lacker and Agathe Soret, \emph{Many-player games of optimal consumption
  and investment under relative performance criteria}, Mathematics and
  Financial Economics \textbf{14} (2020), no.~2, 263--281.

\bibitem{lacker_zariphopoulou_n-agent_nash}
Daniel Lacker and Thaleia Zariphopoulou, \emph{Mean field and $n$-agent games
  for optimal investment under relative performance criteria}, Mathematical
  Finance \textbf{29} (2019), no.~4, 1003--1038.

\bibitem{liu2005option}
Hong Liu and Jiongmin Yong, \emph{Option pricing with an illiquid underlying
  asset market}, Journal of Economic Dynamics and Control \textbf{29} (2005),
  no.~12, 2125--2156.

\bibitem{luo2019nash}
Xiangge Luo and Alexander Schied, \emph{Nash equilibrium for risk-averse
  investors in a market impact game with transient price impact}, Market
  Microstructure and Liquidity \textbf{5} (2019), no.~01n04.

\bibitem{muhle2022stochastic}
Johannes Muhle-Karbe, Zexin Wang, and Kevin Webster, \emph{Stochastic liquidity
  as a proxy for nonlinear price impact}, Operations Research (2023).

\bibitem{pham2009continuous}
Huy{\^e}n Pham, \emph{Continuous-time stochastic control and optimization with
  financial applications}, vol.~61, Springer Science \& Business Media, 2009.

\bibitem{schied2017high}
Alexander Schied, Elias Strehle, and Tao Zhang, \emph{High-frequency limit of
  {N}ash equilibria in a market impact game with transient price impact}, SIAM
  Journal on Financial Mathematics \textbf{8} (2017), no.~1, 589--634.

\bibitem{schied2017state}
Alexander Schied and Tao Zhang, \emph{A state-constrained differential game
  arising in optimal portfolio liquidation}, Mathematical Finance \textbf{27}
  (2017), no.~3, 779--802.

\bibitem{schied2019market}
\bysame, \emph{A market impact game under transient price impact}, Mathematics
  of Operations Research \textbf{44} (2019), no.~1, 102--121.

\bibitem{schoneborn2009liquidation}
Torsten Sch{\"o}neborn and Alexander Schied, \emph{Liquidation in the face of
  adversity: stealth vs. sunshine trading}, EFA 2008 Athens Meetings Paper,
  2009.

\bibitem{webster2023handbook}
Kevin~T. Webster, \emph{Handbook of price impact modeling}, 1. ed., CRC Press,
  Boca Raton, 2023.

\end{thebibliography}

\end{document}